\documentclass{amsart}
\usepackage{amscd}
\usepackage{mathrsfs}
\usepackage{cases}
\usepackage{latexsym}
\usepackage{amsmath}
\usepackage{indentfirst}
\usepackage[arrow,matrix]{xy}
\usepackage{stmaryrd}
\usepackage{amsfonts}
\usepackage{amsmath,amssymb,amscd,bbm,amsthm,mathrsfs,dsfont}
\usepackage{graphicx}
\usepackage{fancyhdr}
\usepackage{amsxtra,ifthen}
\usepackage{verbatim}
\usepackage{latexsym,epsfig,amsmath}
\usepackage{epsfig}
\usepackage{pb-diagram}
\usepackage{youngtab}
\usepackage{pstricks,pst-node}

\DeclareMathOperator{\End}{End} 
\DeclareMathOperator{\Ker}{Ker} \DeclareMathOperator{\Id}{Id}
 \DeclareMathOperator{\sign}{sign}
\DeclareMathOperator{\id}{id} 
 
\DeclareMathOperator{\tr}{tr} 
 
\DeclareMathOperator{\Hom}{Hom}

\numberwithin{equation}{section}

\theoremstyle{plain}
\newtheorem{theorem}{Theorem}[section]
\newtheorem{lemma}[theorem]{Lemma}
\newtheorem{proposition}[theorem]{Proposition}

\theoremstyle{definition}
\newtheorem{definition}[theorem]{Definition}
\newtheorem{example}[theorem]{Example}

\newtheorem{remark}[theorem]{Remark}

\begin{document}

\title[Framed tangle category and invariants of QSG]
{The diagram category of framed tangles and invariants of quantized symplectic group}

\author{Zhankui Xiao}
\address{Xiao: School of Mathematical Sciences, Huaqiao University,
Quanzhou, Fujian, 362021, China}
\email{zhkxiao@hqu.edu.cn}

\author{Yuping Yang}
\address{Yang: School of Mathematics and Statistics, Southwest University, Chongqing, 400715, China}
\email{yupingyang@swu.edu.cn}

\author{Yinhuo Zhang}
\address{Zhang: Department of Mathematics \& Statistics, University of Hasselt, Universitaire Campus, Diepeenbeek, 3590, Belgium}
\email{yinhuo.zhang@uhasselt.be}

\thanks{Xiao was partially supported by the National Natural Science Foundation of China
(No. 11301195), China Scholarship Council and a research foundation of
Huaqiao University (Project 2014KJTD14).}


\begin{abstract}
In this paper we present a categorical version of the first and second fundamental theorems
of the invariant theory for the quantized symplectic groups. Our methods depend on
the theory of braided strict monoidal categories which are pivotal, more explicitly the
diagram category of framed tangles.
\end{abstract}

\maketitle
\section{Introduction}\label{xxsec1}

The fundamental theorems of the classical invariant theory are concerned with
describing generators and relations for invariants of the classical group actions \cite{W}.
There are several but essentially equivalent ways to formulate these fundamental theorems.

Let $G$ be the general (special) linear group, or the
symplectic group, or the orthogonal group. One formulation of the fundamental theorems is
in terms of the endomorphism algebra $\End_{G}(V^{\otimes n})$, where $V$ is the
corresponding natural representation. In this case, the first
fundamental theorem (FFT) describes the endomorphism algebra as the homomorphic image of some
known algebra and the FFT of this formulation is also known as Schur-Weyl duality.
For the general (special) linear group, the endomorphism algebra $\End_{G}(V^{\otimes n})$
is the homomorphic image of the group algebra of the symmetric group following Schur \cite{W}.
For the symplectic or the orthogonal group, the endomorphism algebra $\End_{G}(V^{\otimes n})$
is the homomorphic image of the Brauer algebra with specialized parameters (see
\cite{Br,DDH} for the symplectic case and \cite{Br,DH} for the orthogonal case).

On the other hand, except the case of general (special) linear groups, it is an open
problem for several decades to find a standard form of the second fundamental theorem (SFT)
in the formulation of endomorphism algebras. It needs to describe a suitable ideal in the
Brauer algebra using the standard generators of the Brauer algebra. Recently Hu and the first author
in \cite{HuXiao} proved the SFT for the symplectic group and Lehrer-Zhang in \cite{LZ1}
gave the SFT for the orthogonal group, where they took advantage of
the different versions of invariant theory. Furthermore, the methods in \cite{LZ1}
give rise to the possibility that there should be a unified description
for the different formulations of the fundamental theorems and this goal was achieved
by Lehrer and Zhang in \cite{LZ2} using the category of Brauer diagrams, a symmetric braided
strict monoidal category in the sense of Joyal and Street \cite{JS}. The main motivation of this paper is to
give a quantum analogue of \cite[Theorem 4.8]{LZ2} for the quantized symplectic group.

The invariant theory of quantum groups \cite{Dr,L} in a broad sense has been studied
extensively. An interesting and important aspect of it is the connection with the
topological quantum field theory, i.e. the quantum group theoretical construction of
the Jones polynomial of knots \cite{Jo}. We refer the
reader to the book \cite{Tu} for a comprehensive understanding of this topic.

Let $U_q(\mathfrak{g})$ be the quantized enveloping algebra associated to the
finite dimensional complex simple Lie algebra $\mathfrak{g}$ \cite{CP,L}. The Schur-Weyl duality
formulation of the FFT for quantum invariants describes the endomorphism algebra
$\End_{U_q(\mathfrak{g})}(V^{\otimes n})$, with $V$ the natural representation,
as the homomorphic image of the Hecke algebra of type A \cite{Ji,DPS} or the Birman-Murakami-Wenzl
algebra (BMW algebra for short) with some specialized parameters \cite{Ha,Hu}.
However, in this case, we do not know so much for the SFT. When $q$ is an indeterminant
and $\mathfrak{g}=\mathfrak{sp}_{2m}$, the symplectic Lie algebra, the SFT can be given
from the detailed structure and representations of the BMW algebra \cite{HuXiao,DHS},
where the proofs involve a somewhat technical computation.
Therefore, it is desirable to provide a standard and explicit formulae for the SFT
of the quantized symplectic group and the orthogonal group which has been suggested in
\cite[Section 8]{LZ2}.

This note can be seen as our first attempt, following the idea of Lehrer and Zhang, to
build an explicit connection between different formulations of the fundamental theorems
for the quantum invariant theory \cite{CP,LZZ,L}. We would like to remark that the proof
of \cite[Theorem 4.8]{LZ2} depends on
the linear formulation of the fundamental theorems of the symplectic or the
orthogonal group. However, in quantum case, the linear formulation of
the fundamental theorems do not exist as far as we know.
Our categorical approach here relies highly on the pivotal structure
of braided strict monoidal categories. More precisely, we introduce
the diagram category of framed tangles
\cite{FY}, which is also known as the category of non-directed ribbon graphs \cite{RT,Tu}.
The fact is that the Reshetikhin-Turaev functor (RT-functor for short, see Section \ref{xxsec3}) is pivotal,
a full monoidal functor from the diagram category of framed tangles to the category of tensor
representations of the quantized symplectic group. This fact has been used in the
classical or non-quantum case in \cite{LZ2,RW}.

The contents of this note is organized as follows. In Section \ref{xxsec2} we first introduce
some definitions and basic properties related to the diagram category of framed tangles,
as a strict braided pivotal monoidal category. We then present explicitly, in Section \ref{xxsec3},
the RT-functor between the diagram category of framed tangles and the category
of tenor representations of $U_q(\mathfrak{sp}_{2m})$.
In Section \ref{xxsec4}, we study the fundamental theorems of invariants
for the quantized symplectic group over the rational function field, and prove a quantum analogue
of \cite[Theorem 4.8]{LZ2}.
\vspace{6pt}

\noindent{\bf Acknowledgements.} The work was done while Xiao was visiting the University of Hasselt
and he is very grateful for its hospitality.

\section{Diagram category of Framed Tangles}\label{xxsec2}

In the present section, we introduce some basic definitions and facts about the diagram category
of framed tangles, the BMW algebra.
All the tangles in this paper are assumed to be non-directed.
For the terminologies about monoidal and tensor categories, we refer the
reader to \cite{EGNO} for their precise meanings.

\subsection{Framed tangle category.}\label{xxsec2.1}

Let us first recall some definitions from \cite{FY,Ka}.

\begin{definition}\label{2.1}
A {\em tangle} is a knot diagram inside a rectangle consisting of a finite number
of vertices in the top and the bottom row of the rectangle (not necessarily the same number)
and a finite number of arcs inside the rectangle such that each vertex is connected to another
vertex by exactly one arc, and arcs either connect two vertices or are closed curves.
\end{definition}

\begin{definition}\label{2.2}
Two tangles are {\em regularly isotopic} if they are equivalent by a sequence of the following
Reidemeister Moves II and III
\begin{center}
\epsfig{figure=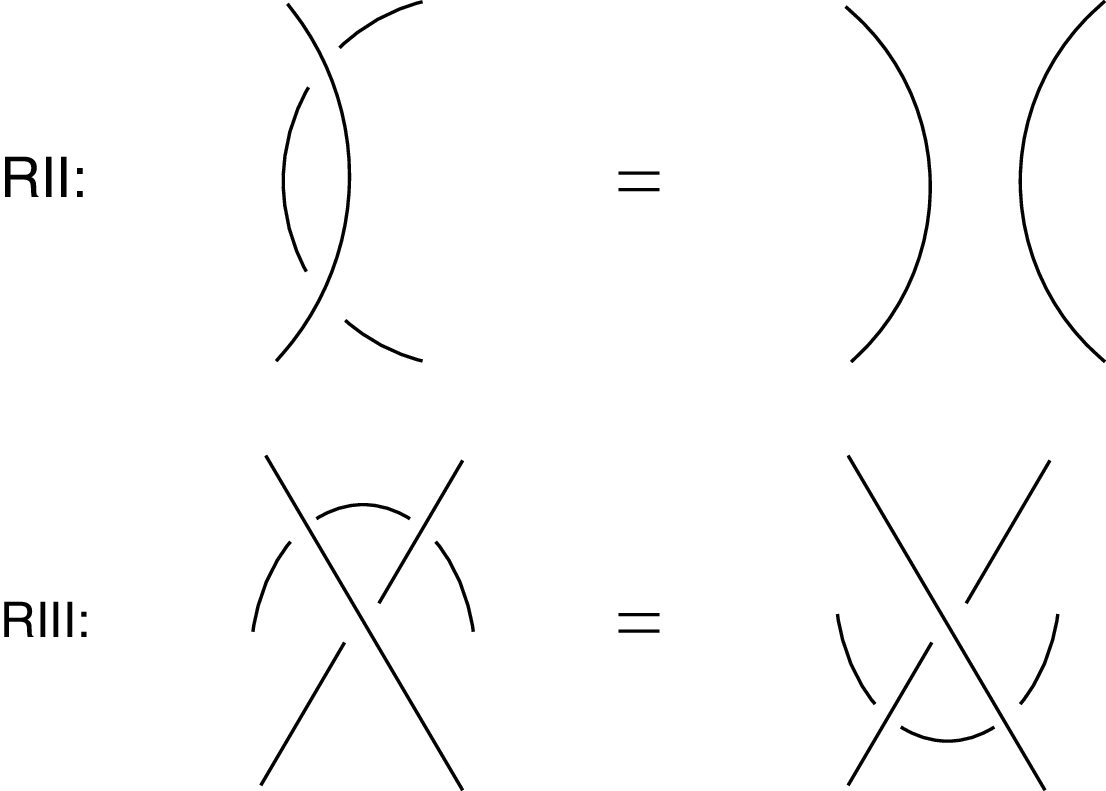,clip,height=3.2cm}
\end{center}
and the isotopies fixing the boundary of the rectangle (or any rotation of them in a local portion
of the rectangle). By $T(s,t)$ we denote the set of all tangles with $s$ vertices
in the top row and $t$ vertices in the bottom row subject to the relations of regular
isotopy and one more relation, named the modified Reidemeister Move I, as follows:
\begin{center}
\epsfig{figure=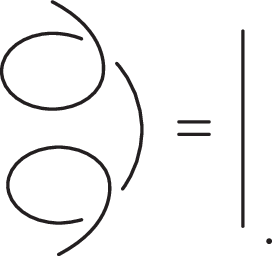,clip,height=1.8cm}
\end{center}
\end{definition}

There are two operations on tangles:
\begin{eqnarray}\label{eq2.1}
\circ:& & T(s,t)\times T(t,l)\longrightarrow T(s,l),\\
\otimes :& & T(s,t)\times T(l,m)\longrightarrow T(s+l,t+m).
\end{eqnarray}
The {\em composition} $\circ$ is defined by concatenation of tangles, reading
from up to down for our late convenience,
and the {\em tensor product} $\otimes$ is given by juxtaposition of tangles. More explicitly,
$D\otimes D'$ means placing $D'$ on the right of $D$ without overlapping.
In order to obtain the explicit definition of the framed tangle category, we need to add
an element $id_0$ in $T(0,0)$ and define the corresponding composition and tensor product as follows:
for any $f\in T(0,t),$ $g\in T(s,0)$ and $h\in  T(s,t),$
\begin{eqnarray}\label{eq2.4}
&&id_0\circ f=f,\ \  g\circ id_0=g,\\
&&h\otimes id_0=id_0\otimes h=h.
\end{eqnarray}

With these notations, we can give the definition of the framed tangle category
following \cite[Definition 3.8]{FY}.

\begin{definition}\label{2.3}
The {\em framed tangle category}, denoted $\mathcal{FT}$, is a strict monoidal category with objects
$\mathbb{N}=\{0,1,2,\ldots\}$ and morphisms given by $\Hom_{\mathcal{FT}}(s,t)=T(s,t)$.
The tensor product on objects is defined by $m\otimes n:=m+n$, and the composition and tensor product
on morphisms are defined by (\ref{eq2.1}-2.4).
\end{definition}

The framed tangle category $\mathcal{FT}$ is a free pivotal strict monoidal category with
a self-dual generator, see \cite[Theorem 3.6]{FY}. The duality functor $^*$ is a contravariant
functor taking each object to itself, and
acting on morphisms by rotating the rectangle containing a tangle through $\pi$ about an axis
perpendicular to the rectangle. It is clear that $^{**}$ is just the identity functor $\Id_{\mathcal{FT}}$,
see \cite{FY} for more details. Here we remark that $T(s,t)=\emptyset$
by definition if $s+t$ is an odd number.

\subsection{Diagram category of framed tangles}\label{xxsec2.2}

To meet our need, we give the linear version of the framed tangle category.
Let $\mathbb{K}$ be the rational function field $\mathbb{Q}(r,q)$ with $r,q$ being indeterminates.
The linear framed tangle category, denoted $\mathbb{K}(\mathcal{FT})$, consists of the same objects of
$\mathcal{FT}$, but $\Hom_{\mathbb{K}(\mathcal{FT})}(s,t)$ is the $\mathbb{K}$-linear span of $T(s,t)$.
The composition and the tensor product of morphisms are given by bilinear extensions of those in $\mathcal{FT}$.
The diagram category of framed tangles (see below Definition \ref{2.5}) can be obtained as a quotient of
$\mathbb{K}(\mathcal{FT})$ subject to certain relations on morphisms.

\begin{definition}\label{2.4}
For $s,t\in \mathbb{N}$ we denote by $\mathcal{D}(s,t)$ the $\mathbb{K}$-linear span of $T(s,t)$ subject to the following relations:
\begin{center}
\epsfig{figure=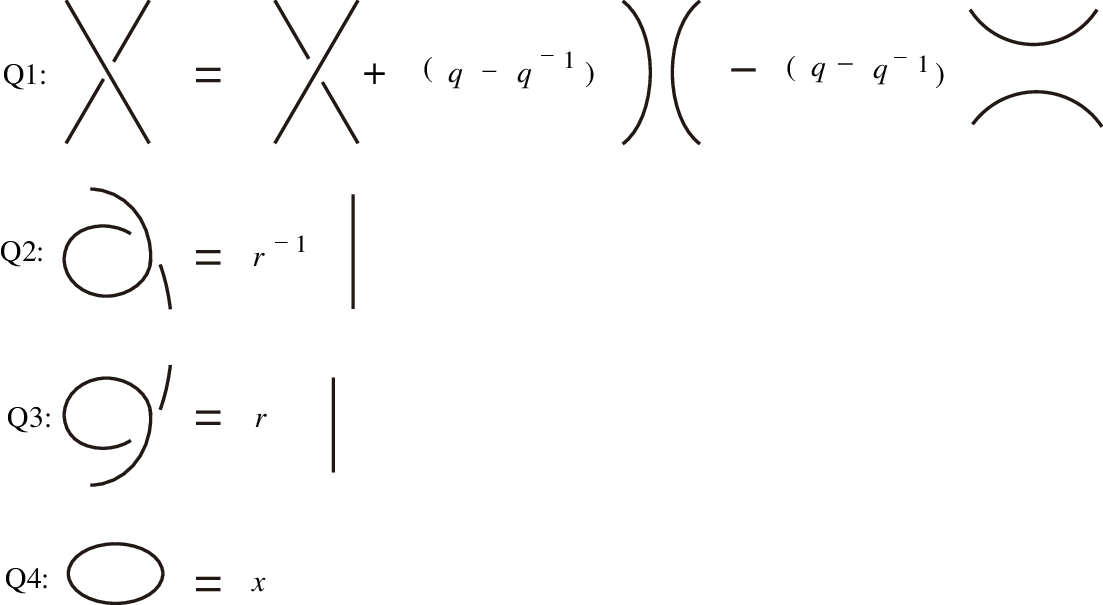,clip,height=6.0cm}
\end{center}
where $x=1+\frac{r-r^{-1}}{q-q^{-1}}.$
\end{definition}

We shall explain the composition defined by (\ref{eq2.1}) can be extended to $\mathcal{D}(s,t)$. According to \cite[Theorem 3.5]{FY}, any tangle in $T(s,t)$ can be generated through compositions and tensor products by the following five tangles:
\begin{center}
\begin{picture}(280, 40)(-5,0)
\put(0, 0){\line(0, 1){40}}
\put(5, 0){,}

\put(40, 0){\line(1, 2){8}}
\put(60, 0){\line(-1, 2){20}}
\put(60, 40){\line(-1, -2){8}}
\put(65, 0){,}

\put(100, 40){\line(1, -2){8}}
\put(100, 0){\line(1, 2){20}}
\put(120, 0){\line(-1, 2){8}}
\put(125, 0){,}

\qbezier(160, 0)(175, 60)(190, 0)
\put(195, 0){,}

\qbezier(230, 30)(245, -30)(260, 30)
\put(265, 0){.}
\end{picture}
\end{center}
We shall refer to these generators as the {\em elementary} tangles and denote them by $I, X, X^{op}, A, U$
respectively. In order to show that the composition can be extended to $\mathcal{D}(s,t)$,
We only need to show that when the terms in both handsides of the relations Q1-Q4
compose with an elementary tangle, the results are still identities.
This follows from a direct verification. It is obvious that there is a natural tensor operation over
$\mathcal{D}(s,t)$. Therefore we can now give the following new definition, which is a little stronger
than the one in \cite[Definition 4.1.2]{FY}.

\begin{definition}\label{2.5}
The {\em diagram category of framed tangles} with parameters $r,q$, denoted $\mathbb{D}(r,q)$,
is a strict monoidal category with objects $\mathbb{N}=\{0,1,2,\ldots\}$ and morphisms
$\Hom_{\mathbb{D}(r,q)}(s,t)=\mathcal{D}(s,t)$. The composition $\circ$ and the tensor product
$\otimes$ are inherited from those of $\mathbb{K}(\mathcal{FT})$.
\end{definition}

From now on we only focus us on the diagram category of framed tangles. The category
$\mathbb{D}(r,q)$ has a contravariant functor $^*:\mathbb{D}(r,q)\to \mathbb{D}(r,q)$
such that $n^*=n$ for any object $n\in\mathbb{N}$. In order to get a clear picture of morphisms
under the functor $^*$,
we define the following useful tangles with notations compatible with those in the general
theory of tensor category \cite{EGNO}. Let $U_n: n\otimes n\to 0$ and $A_n: 0\to n\otimes n$
be tangles defined by
\begin{eqnarray}\label{eq2.5}
U_n&=&(I^{\otimes (n-1)}\otimes U\otimes I^{\otimes(n-1)})\circ\cdots\circ (I\otimes U\otimes I)\circ U,\\
A_n&=&A\circ(I\otimes A\otimes I)\circ\cdots\circ (I^{\otimes (n-1)}\otimes A\otimes I^{\otimes (n-1)}).
\end{eqnarray}
These are depicted as follows
\begin{center}
\begin{picture}(250, 40)(-5,0)
\put(0, 10){$U_n=$}
\qbezier(25, 30)(55, -35)(85, 30)
\put(32, 25){\tiny$n$}
\put(32, 20){...}
\qbezier(40, 30)(55, -15)(70, 30)
\put(75, 0){,}

\put(125, 10){$A_n=$}
\qbezier(150, 0)(180, 60)(210, 0)
\put(158, 10){...}
\put(158, 0){\tiny$n$}
\qbezier(165, 0)(180, 40)(195, 0)
\put(215, 0){.}

\end{picture}
\end{center}
Let $I_n:=I^{\otimes n}$ be the identity morphism $\id_n$. The linear mapping
$^*: \mathcal{D}(s,t)\to \mathcal{D}(t,s)$
defined for any tangle $D\in \mathcal{D}(s,t)$ by $D^*:=(I_t\otimes A_s)\circ
(I_t\otimes D\otimes I_s)\circ (U_t\otimes I_s)$ can be described as follows
\begin{center}
\begin{picture}(100, 60)(-20,0)
\put(-5, 20){\line(1, 0){35}}
\put(-5, 20){\line(0, 1){20}}
\put(30, 20){\line(0, 1){20}}
\put(-5, 40){\line(1, 0){35}}
\put(8, 25){$D$}

\qbezier(5, 40)(40, 95)(60, 0)
\qbezier(20, 40)(35, 70)(45, 0)
\put(47, 10){...}

\qbezier(5, 20)(-15, -10)(-20, 60)
\qbezier(20, 20)(-20, -35)(-35, 60)
\put(-30, 50){...}
\put(65, 0){.}
\end{picture}
\end{center}

By \cite[Lemma 3.3]{FY}, we have $^{**}=\Id$, the identity functor. Moreover, the following compositions
\[\xymatrix@C=1.5cm{n\ar[r]^{A_n\otimes I_n \ \ \ \  }& n\otimes n\otimes n\ar[r]^{\ \ \ \ I_n\otimes U_n}& n ,}\]
\[\xymatrix@C=1.5cm{n\ar[r]^{I_n\otimes A_n \ \ \ \ \ }&n\otimes n\otimes n\ar[r]\ar[r]^{\ \ \ \  \ U_n\otimes I_n}& n}\]
are both the identity morphism. Hence the contravariant functor $^*$ gives rise to the duality of
$\mathbbm{D}(r,q)$ with evaluation $U$ and coevaluation $A$. Hence we have

\begin{proposition}\label{2.6}
The diagram category of framed tangles $\mathbb{D}(r,q)$ is pivotal.
\end{proposition}

The following analogue of \cite[Theorem 2.6]{LZ2} to some extent is known for experts
(see \cite{FY,RT,Tu}). However, we haven't found it in literatures, and hence sketch
the proof here for the completeness.

\begin{theorem}\label{2.7}
Any morphism of $\mathbbm{D}(r,q)$ is generated by four elementary tangles
$I, X, A, U$ through linear combination, composition and tensor product.
A complete set of relations among these four generators is given by the following identities and their dualities:
\begin{eqnarray}\label{eq2.7}
&I\circ I=I,\ (I\otimes I)\circ X=X,\ A\circ (I\otimes I)=A,&\\
&X\circ X=I\otimes I+(q-q^{-1})X-r^{-1}(q-q^{-1})U\circ A,&\\
&(X\otimes I)\circ(I\otimes X)\circ(X\otimes I)=(I\otimes X)\circ (X\otimes I)\circ (I\otimes X),&\\
&A\circ X=r^{-1}A,&\\
&A\circ U=x,&\\
&(A\otimes I)\circ (I\otimes X)=(I\otimes A)\circ (\big(X+(q^{-1}-q)(I\otimes I-U\circ A)\big)\otimes I),&\\
&(A\otimes I)\circ (I\otimes U)=I.&
\end{eqnarray}
\end{theorem}

\begin{proof}
Since any tangle can be generated by $I, X, X^{op}, A, U$ through the composition and the tensor product,
according to \cite[Theorem 3.5]{FY}, any morphism in $\mathcal{D}(s,t)$ can be obtained
by $I, X, A, U$ through linear combination, composition and tensor product because of
the relation Q1. Thus, we proved the first claim of the theorem.

For the second part of the theorem, we first explain how to obtain the identities (\ref{eq2.7}-2.11)
and their dualities. The identities (2.7) and their dualities are obvious. The identity (2.8)
can be obtained from the relation Q1 by composing with elementary tangle $X$, using
the Reidemeister move RII and the identity Q2.
The identities (2.9-2.11) are corresponding to the Reidemeister move RIII, Q2 and Q4 respectively.
Finally, (2.12-2.13) are deduced from the singular isotopy, called sliding and straightening respectively,
see \cite[Definition 2.3]{FY}. Now we need to prove that these identities (\ref{eq2.7}-2.13) are complete.
For any $s,t\in \mathbb{N}$, it follows from the definition of $\mathcal{D}(s,t)$
that all relations among the four generators are determined by the regular isotopy
and the relations Q1-Q4. Set $X^{op}:=X-(q-q^{-1})I\otimes I-(q-q^{-1})U\circ A$.
Then the regular isotopy, the modified Reidemeister move I, and the relations Q1-Q4 can be deduced from
the identities (\ref{eq2.7}-2.13) and their dualities.
\end{proof}

For any object $n>0$, the set of morphisms $\mathcal{D}(n,n)$ forms a unital
associative $\mathbb{K}$-algebra under the composition of tangles. This is the
Kauffman's tangle algebra \cite{Ka} which is isomorphic to a BMW algebra. We express BMW
algebra by generators and relations for late use.

\begin{definition}{\rm (\cite{BW,Mu})} \label{2.8}
The generic \emph{BMW algebra} $B_n(r,q)$ is
a unital associative $\mathbb{K}$-algebra generated by the
elements $T_i^{\pm1}$ and $E_i$ for $1\leq i\leq n-1$ subject to the
relations:
  \begin{alignat}{2}
    T_i-T_i^{-1}&=(q-q^{-1})(1-E_i), &\quad&\text{ for } 1\leq i\leq n-1,\\
    E_i^2&=x E_i, &&\text{ for } 1\leq i\leq n-1,\\
    T_iT_{i+1} T_i&=T_{i+1} T_iT_{i+1}, &&\text{ for } 1\leq i\leq n-2,\\
    T_iT_j&=T_jT_i, &&\text{ for }|i-j|>1,\\
    E_iE_{i+1} E_i&=E_i,\;  E_{i+1}E_i E_{i+1}=E_{i+1},
    &&\text{ for } 1\leq i\leq n-2,\\
    T_iT_{i+1}E_i&=E_{i+1} E_i,\; T_{i+1}T_iE_{i+1}=E_iE_{i+1}, &&\text{
      for } 1\leq i\leq n-2,\\
    E_iT_i&=T_iE_i=r^{-1}E_i &&\text{ for } 1\leq i\leq n-1,\\
    E_iT_{i+1}E_i&=rE_i,\;E_{i+1}T_iE_{i+1}=rE_{i+1}, &&\text{ for }
    1\leq i\leq n-2,
      \end{alignat}
where $x=1+\dfrac{r-r^{-1}}{q-q^{-1}}$.
\end{definition}

\begin{theorem} \label{2.9}
Kauffman's tangle algebra $\mathcal{D}(n,n)$ is isomorphic to the BMW algebra $B_n(r,q)$
with the isomorphism given by
$$
T_i\mapsto I_{i-1}\otimes X\otimes I_{n-1-i},\ \ E_i\mapsto I_{i-1}\otimes (U\circ A)\otimes I_{n-1-i}.
$$
\end{theorem}

\section{The Reshetikhin-Turaev Functor}\label{xxsec3}

In this section, we present explicitly the RT-functor between the diagram category of framed tangles
and the category of tenor representations in the symplectic case \cite{RT}.
Let $m\in \mathbb{N}$, and $q$ an indeterminate. Let $U_q(\mathfrak{sp}_{2m})$ be the
quantized enveloping algebra over $\mathbb{Q}(q)$ associated to $\mathfrak{sp}_{2m}$
(see \cite{L}). Let $V$ be the natural representation of $U_q(\mathfrak{sp}_{2m})$.
Then $\dim V=2m$ equipped with a skew symmetric bilinear form $(-,-)$. We refer the
reader to \cite[Section 2]{Hu} for the explicit action of $U_q(\mathfrak{sp}_{2m})$
on the tensor space $V^{\otimes n}$.

There is a right action of the BMW algebra $B_n(-q^{2m+1},q)$, i.e. $r$ specialized to $-q^{2m+1}$,
on the tensor space $V^{\otimes n}$,
which we now recall. For each integer $i$ with
$1\leq i\leq 2m$, set $i':=2m+1-i$. We fix an ordered basis
$\{v_i\}_{i=1}^{2m}$ of $V$ such that
$$
(v_i,v_j)=0=(v_{i'},v_{j'}),\quad
(v_i,v_{j'})=\delta_{ij}=-(v_{j'},v_i),\ \ \forall 1\leq i,j\leq m.
$$
We set
$$
(\rho_1,\cdots,\rho_{2m}):=(m,m-1,\cdots,1,-1,\cdots,-m+1,-m),
$$
and $\epsilon_i:=\sign(\rho_i)$. For any $i,j\in\{1,2,\cdots,2m\}$,
we use $E_{i,j}\in\End_{\mathbb{Q}(q)}(V)$ to denote the elementary matrix whose
entries are all zero except $1$ for the $(i,j)$-th entry.
Let us define (see \cite[Section 3]{Hu} for the same notations)
$$\begin{aligned} \beta' &:=\sum_{1\leq i\leq
2m}\Bigl(qE_{i,i}\otimes E_{i,i}+q^{-1}E_{i,i'}\otimes
E_{i',i}\Bigr)+\sum_{\substack{1\leq
i,j\leq 2m\\ i\neq j,j'}} E_{i,j}\otimes E_{j,i}+\\
&\qquad\qquad (q-q^{-1})\sum_{1\leq i<j\leq 2m}\Bigl(E_{i,i}\otimes
E_{j,j}-q^{\rho_j-\rho_i}\epsilon_i\epsilon_j E_{i,j'}\otimes
E_{i',j}\Bigr),\\
\gamma' &:=\sum_{1\leq i,j\leq
2m}q^{\rho_j-\rho_i}\epsilon_i\epsilon_j E_{i,j'}\otimes E_{i',j}.
\end{aligned}
$$
Note that the operators $\beta', \gamma'$ are related to each other by
the equation
$$\beta'-(\beta')^{-1}=(q-q^{-1})(\id_{V^{\otimes
2}}-\gamma').
$$
For $i=1,2,\ldots,n-1$, we set
$$
\beta'_i:=\id_{V^{\otimes i-1}}\otimes \beta'\otimes \id_{V^{\otimes n-i-1}},\quad
\gamma'_i:=\id_{V^{\otimes i-1}}\otimes \gamma'\otimes \id_{V^{\otimes n-i-1}}.
$$
By \cite[(10.2.5)]{CP} and \cite[Section 4]{Ha}, the mapping which sends the generator
$T_i$ to $\beta'_i$ and $E_i$ to $\gamma'_i$ can be naturally extended to a right action
of $B_n(-q^{2m+1},q)$ on the tensor space $V^{\otimes n}$. This right action commutes
with the left action of $U_q(\mathfrak{sp}_{2m})$ on $V^{\otimes n}$.

In order to define the RT-functor explicitly in our case, we need the following
$\mathbb{Q}(q)$-linear maps:
$$\begin{aligned}
\check{R}: V\otimes V\longrightarrow V\otimes V,&\quad\quad v\otimes w\mapsto \beta'(v\otimes w),\\
C: \mathbb{Q}(q)\longrightarrow V\otimes V,&\quad\quad 1\mapsto \alpha:=\sum_{1\leq k\leq
2m}q^{-\rho_k}\epsilon_k v_k\otimes v_{k'},\\
E: V\otimes V\longrightarrow \mathbb{Q}(q),&\quad\quad v_i\otimes v_j\mapsto q^{-\rho_i}\epsilon_j(v_i,v_j).&
\end{aligned}
$$
Clearly $\gamma'=E\circ C$, where the composition reads from left to right,
because of the {\em right} action of $B_n(-q^{2m+1},q)$ on $V^{\otimes n}$.
By the above statements, the maps $\check{R}, C$ and $E$ are all
$U_q(\mathfrak{sp}_{2m})$-homomorphisms. They have the following properties.

\begin{lemma}\label{xx3.1}
Denote the identity map on $V$ by $\id$. Then the maps $\check{R}, C$ and $E$ satisfy the relations:
\begin{eqnarray}\label{eq3.1}
&\check{R}^2=\id^{\otimes 2}+(q-q^{-1})(\check{R}-r^{-1}E\circ C),&\\
&(\check{R}\otimes \id)\circ(\id\otimes \check{R})\circ(\check{R}\otimes \id)=
(\id\otimes \check{R})\circ (\check{R}\otimes \id)\circ (\id\otimes \check{R}),&\\
&C\circ \check{R}=r^{-1}C,&\\
&C\circ E=x,&\\
&(C\otimes \id)\circ (\id\otimes \check{R})=(\id\otimes C)\circ (\check{R}^{-1}\otimes \id),&\\
&(C\otimes \id)\circ (\id\otimes E)=\id,&
\end{eqnarray}
where $r=-q^{2m+1}$ and the composition reads from left to right.
\end{lemma}

\begin{proof}
Note that the mapping which sends each
$T_i$ to $\beta'_i$ and each $E_i$ to $\gamma'_i$ can be naturally extended to a right action
of $B_n(-q^{2m+1},q)$ on the tensor space $V^{\otimes n}$. Then $\check{R}$ is invertible and
$\check{R}^{-1}=\check{R}+(q^{-1}-q)(\id^{\otimes 2}-E\circ C)$. The identities (\ref{eq3.1}-3.2)
now follow from Definition \ref{2.8} and the remaining relations can be verified directly.
\end{proof}

\begin{definition}\label{2.10}
We denote by $\mathcal{T}(V)$ the full subcategory of $U_q(\mathfrak{sp}_{2m})$-modules
with objects $V^{\otimes n}$, where $n\in\mathbb{N}$ and $V^{\otimes 0}=\mathbb{Q}(q)$
by convention. The usual tensor product of $U_q(\mathfrak{sp}_{2m})$-modules and of
$U_q(\mathfrak{sp}_{2m})$-homomorphisms is a bi-functor $\mathcal{T}(V)\times \mathcal{T}(V)
\to \mathcal{T}(V)$, which will be called the tensor product of this category. We call
$\mathcal{T}(V)$ the {\em category of tensor representations} of $U_q(\mathfrak{sp}_{2m})$.
\end{definition}

Since $V\cong V^*$ as $U_q(\mathfrak{sp}_{2m})$-modules, the category $\mathcal{T}(V)$
is also a pivotal strict monoidal category with the evaluation $E$ and the coevaluation $C$.
Under the duality decided by $E$ and $C$, $V$ and $V^*$ are the same object in the category $\mathcal{T}(V)$.
Furthermore, $\mathcal{T}(V)$ has a (non-symmetric) braiding structure given by an $R$-matrix
of $U_q(\mathfrak{sp}_{2m})$, see \cite{LR,RT}.

Let $r=-q^{2m+1}$ and $\mathbb{D}(\mathfrak{sp}_{2m}):=\mathbb{D}(r,q)$.
We have the following results.

\begin{theorem}\label{2.11}
There exists a unique additive covariant functor $F: \mathbb{D}(\mathfrak{sp}_{2m})\rightarrow
\mathcal{T}(V)$ satisfying the following properties:

\begin{enumerate}
\item[(i)] $F$ sends the object $n$ to $V^{\otimes n}$ and the morphism $D: s\to t$ to $F(D):
V^{\otimes s}\to V^{\otimes t}$, where $F(D)$ is defined on the generators of tangles by
$$\begin{aligned}
F\left(
\begin{picture}(30, 20)(0,0)
\put(15, -15){\line(0, 1){35}}
\end{picture}\right)=\id_V,
\quad\quad&
F\left(
\begin{picture}(30, 20)(0,0)
\put(5, -15){\line(1, 2){8}}
\put(25, -15){\line(-1, 2){18}}
\put(24, 21){\line(-1, -2){7}}
\end{picture}\right) = \check{R}, \\
F\left(
\begin{picture}(30, 20)(0,0)
\qbezier(5, -15)(15, 50)(25, -15)
\end{picture}\right) = C,
\quad\quad&
F\left(
\begin{picture}(30, 20)(0,0)
\qbezier(5, 20)(15, -50)(25, 20)
\end{picture}\right) = E.
\end{aligned}$$
\item[(ii)] $F$ is a pivotal monoidal functor, i.e. $F$ preserves the tensor products and the dualities.
\end{enumerate}
\end{theorem}

\begin{proof}
Since the linear maps $\check{R}, C$ and $E$ are $U_q(\mathfrak{sp}_{2m})$-homomorphisms,
by Theorem \ref{2.7} it is clear that $F$ preserves tensor products and dualities of
objects, i.e. $F(s\otimes t)=F(s)\otimes F(t)$ and $F(n^*)=F(n)=V^{\otimes n}=F(n)^*$. As
a covariant functor, $F$ preserves the composition of tangles, and by (ii) $F$ preserves
the tensor products and the dualities of morphisms.

Thanks to \cite[Theorem 5.1]{RT} we just need to check that $F$ is well-defined.
In fact, it is enough to show that the images of the generators of tangles satisfy the relations
in Theorem \ref{2.7}. This is clear from Lemma \ref{xx3.1}.
\end{proof}

Following \cite[\S 8.3]{LZ2}, the functor $F$ in Theorem \ref{2.11} is
called the {\em Reshetikhin-Turaev functor} (RT-functor for short as seen in the introduction).
Now let us end this section with a technical lemma.

\begin{lemma}\label{2.12}
Let $H(s,t):=\Hom_{U_q(\mathfrak{sp}_{2m})}(V^{\otimes s},V^{\otimes t})$ for all $s,t\in\mathbb{N}$.
Define the linear maps
$$
\mathbb{U}_s^t:=(-\otimes I_t)\circ (I_s\otimes U_t): \mathcal{D}(n,s+t)\longrightarrow \mathcal{D}(n+t,s),
$$
$$
\mathbb{A}_t^n:=(I_{n}\otimes A_t)\circ (-\otimes I_t): \mathcal{D}(n+t,s)\longrightarrow \mathcal{D}(n,s+t).
$$
Then we have:

\begin{enumerate}
\item[(i)] the $\mathbb{Q}(q)$-linear maps
$$
F\mathbb{U}_s^t:=(-\otimes \id^{\otimes t}_V)(\id^{\otimes s}_V\otimes F(U_t)): H(n,s+t)\longrightarrow H(n+t,s),
$$
$$
F\mathbb{A}_t^n:=(\id^{\otimes n}_V\otimes F(A_t))(-\otimes \id^{\otimes t}_V): H(n+t,s)\longrightarrow H(n,s+t)
$$
are well-defined and are mutually inverses of each other;

\item[(ii)] the functor $F$ induces a linear map
$$
F^s_t: \mathcal{D}(s,t)\longrightarrow H(s,t)=\Hom_{U_q(\mathfrak{sp}_{2m})}(V^{\otimes s},V^{\otimes t}),\quad
D\mapsto F(D),
$$
and the following diagrams are commutative:
\[\xymatrix@C=1.0cm{ \mathcal{D}(n+t,s)\ar[r]^{\mathbb{A}_t^n}
\ar[d]_{F_s^{n+t}}& \mathcal{D}(n,s+t)\ar[d]^{F_{s+t}^n}& \\
  H(n+t,s)\ar[r]^{F\mathbb{A}_t^n} &H(n,s+t),}
    \xymatrix@C=1.0cm{ \mathcal{D}(n,s+t)\ar[r]^{\mathbb{U}_s^t}
\ar[d]_{F_{s+t}^n}& \mathcal{D}(n+t,s)\ar[d]^{F_s^{n+t}}& \\
  H(n,s+t)\ar[r]^{F\mathbb{U}_s^t} &H(n+t,s).}\]
\end{enumerate}
\end{lemma}

\begin{proof}
First of all, we show that $\mathbb{A}_t^n$ and $\mathbb{U}_s^t$
are mutually inverses of each other. Let $D\in \mathcal{D}(n,s+t)$. Then we have
\begin{eqnarray*}
\mathbb{A}_t^n\mathbb{U}_s^t(D)&=&(I_n\otimes A_t)\circ(D\otimes I_t\otimes I_t)\circ(I_s\otimes U_t\otimes I_t)\\
&=&D\circ(I_s\otimes I_t\otimes A_t)\circ (I_s\otimes U_t\otimes I_t)\\
&=&D\circ (I_s\otimes I_t)\\
&=&D.
\end{eqnarray*}
Hence $\mathbb{A}_t^n\mathbb{U}_s^t$ is the identity morphism of $\mathcal{D}(n,s+t)$.
Similarly one can show that $\mathbb{U}_s^t\mathbb{A}_t^n$ is the identity morphism
of $\mathcal{D}(n+t,s)$ for each $n\in \mathbb{N}$.
Now the Claim (i) follows from Theorem \ref{2.11}.

Since the RT-functor $F$ preserves both the composition and the tensor product of tangles,
it is obvious that $F(\mathbb{A}_t^n(D))=F\mathbb{A}_t^n(F(D))$ for any tangle $D\in \mathcal{D}(n+t,s)$.
Hence we obtain the commutativity of the first diagram. Similarly, we have the second one.
\end{proof}

\section{Fundamental Theorems for Quantized Symplectic Group}\label{xxsec4}

In this section we present a categorical version of the fundamental theorems of invariants
for the quantized symplectic group. For each positive integer $n$ we introduce the quantum integer
$$
[n]:=\frac{q^n-q^{-n}}{q-q^{-1}},\quad\text{and}\quad [n]!:=[n][n-1]\cdots [1],
$$
with convention $[0]:=1$.

\begin{definition}\label{3.1}
For $k\in \mathbb{N}$ and $1\leq i<n$ define $Y_i(k)\in \mathcal{D}(n,n)=B_n(r,q)$ by
$$
Y_i(k):=\frac{-1}{[k+1]}\left([k]T_i-q^k+\frac{q^k-q^{-k}}{1+rq^{-2k+1}}E_i  \right).
$$
\end{definition}

\begin{proposition}\label{3.2}
The elements satisfy the following relations:
$$
\begin{aligned}
Y_i(k)Y_{i+1}(k+h)Y_i(h)&=Y_{i+1}(h)Y_i(k+h)Y_{i+1}(k),\\
Y_i(k)Y_j(h)&=Y_j(h)Y_i(k)\quad \text{for}\ |i-j|>1.
\end{aligned}
$$
\end{proposition}

\begin{proof}
The second identity is clear. For the first relation, we can assume $i=1$
without lose of the generality. Using the relations of
BMW algebra and Theorem \ref{2.9}, it can be verified via a direct but complicated calculation in
the $15$-dimensional algebra $B_3(r,q)$
\end{proof}

Let $\mathfrak{S}_n$ be the symmetric group on $n$ words with the standard Coxter generators $s_1, s_2,
\ldots, s_{n-1}$. To a reduced expression $s_{i_1}s_{i_2}\cdots s_{i_l}$, we associate an element
$Y(s_{i_1}s_{i_2}\cdots s_{i_l})\in B_n(r,q)$ of the form:
$$
Y(s_{i_1}s_{i_2}\cdots s_{i_l}):=Y_{i_1}(k_1)Y_{i_2}(k_2)\cdots Y_{i_l}(k_l),
$$
where the integers $k_1,\ldots,k_l$ are determined as follows. Firstly we identify
each element of $\mathfrak{S}_n$ with an $n$-string diagram as usual and label the
vertices of an $n$-string diagram in the top and bottom row by the indices $1,2,\ldots,n$
from left to right. Under this identification, we draw the $n$-string diagram of the
reduced expression $s_{i_1}s_{i_2}\cdots s_{i_l}$ (from up to down) such that each factor $s_{i_j}$,
$1\leq j\leq l$, corresponds to an individual $n$-string diagram and to a crossing of two adjacent
strings. Since the word is reduced, any two strings can cross at most once. We label
the crossings by their corresponding position in the reduced expression, and label each
string with the number of its starting vertex on the most top of the diagram. Then for
the $j$-th crossing, if the numbers of the two strings are $a_j$ and $b_j$ with $a_j<b_j$,
we set $k_j=b_j-a_j$.

The element $Y(s_{i_1}s_{i_2}\cdots s_{i_l})$ will be called the {\em Yang-Baxter element}
associated to the reduced expression $s_{i_1}s_{i_2}\cdots s_{i_l}$.

\begin{example}\label{3.3}
The $5$-string diagram corresponding to the reduced expression $s_3s_2s_1s_3s_4$ is
\begin{center}
\epsfig{figure=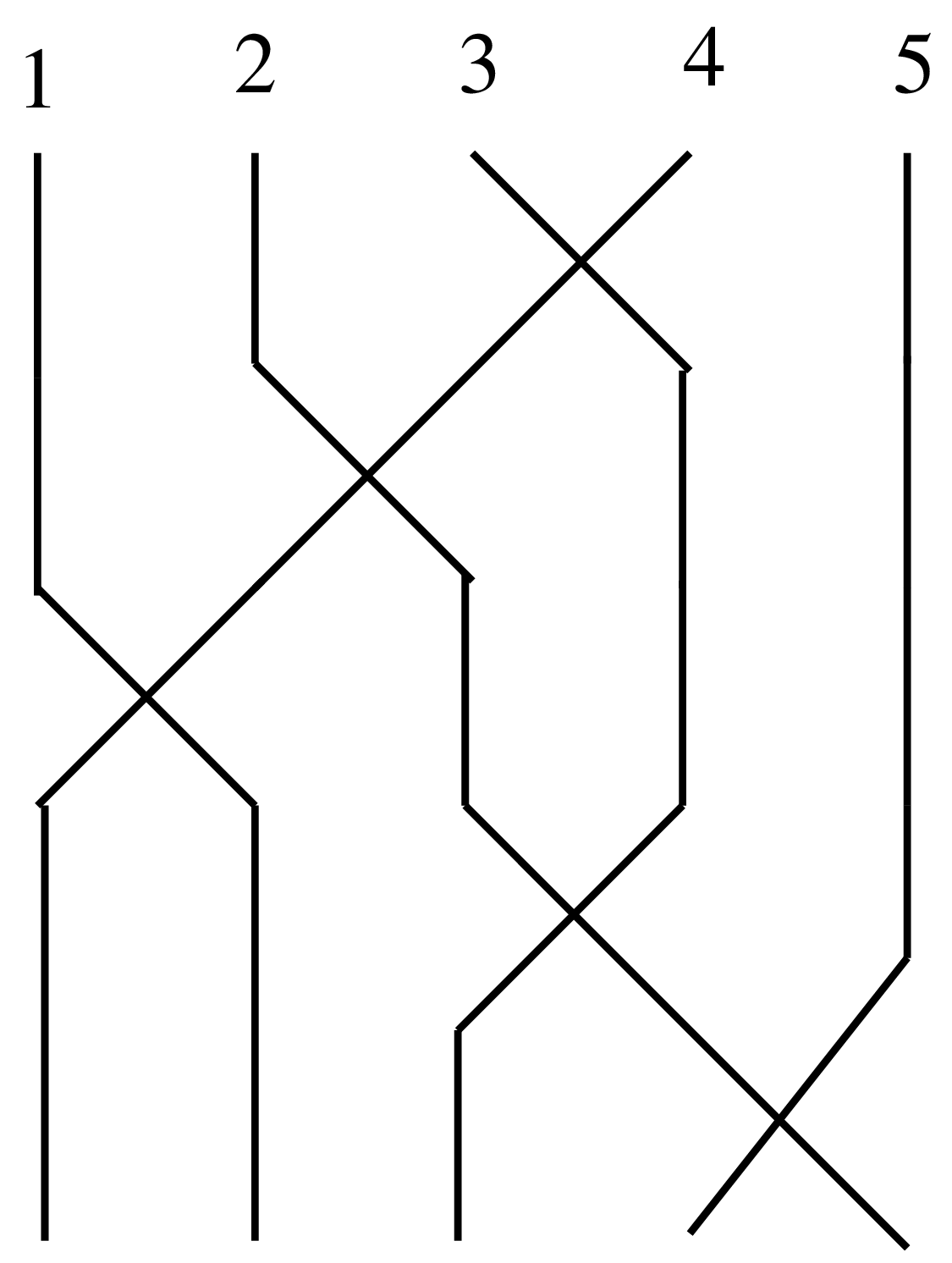,clip,height=3.5cm}
\end{center}
and hence the Yang-Baxter element associated to it is
$$
Y(s_3s_2s_1s_3s_4)=Y_3(1)Y_2(2)Y_1(3)Y_3(1)Y_4(3).
$$
\end{example}

\begin{lemma}\label{3.4}
If two reduced expressions represent the same permutation, then the
associated Yang-Baxter elements in $B_n(r,q)$ are the same.
\end{lemma}

\begin{proof}
It follows from the Matsumoto's theorem \cite{Ma} that two reduced expressions represent
the same permutation if and only if they are related by a finite sequence of braid moves
of the form $s_is_{i+1}s_i\rightarrow s_{i+1}s_is_{i+1}$, $s_{i+1}s_is_{i+1}\rightarrow
s_is_{i+1}s_i$, and $s_is_j=s_js_i$ for $|i-j|>1$. Let us consider the braid move
$s_is_{i+1}s_i\rightarrow s_{i+1}s_is_{i+1}$ and the other two types of braid moves can
be proved similarly. Let $w=w_1s_is_{i+1}s_iw_2$ be a reduced expression with subword
$s_is_{i+1}s_i$ for some $1\leq i<n$. Assume $(i)w^{-1}_1=a$, $(i+1)w^{-1}_1=b$ and
$(i+2)w^{-1}_1=c$. Then $a<b<c$ since $w$ and $w_1$ are reduced. Hence the Yang-Baxter element
associated to $w_1s_is_{i+1}s_iw_2$ contains subword $Y_i(b-a)Y_{i+1}(c-a)Y_i(c-b)$ corresponding
to the subword $s_is_{i+1}s_i$. On the other hand, the Yang-Baxter element
associated to $w_1s_{i+1}s_is_{i+1}w_2$ contains subword $Y_{i+1}(c-b)Y_i(c-a)Y_{i+1}(b-a)$
corresponding to the subword $s_{i+1}s_is_{i+1}$. Now Proposition \ref{3.2}
completes the proof of this lemma.
\end{proof}

Lemma \ref{3.4} tells us that the Yang-Baxter element can be defined for each permutation,
independent of its explicit reduced expressions.
Let $Y_n\in B_n(r,q)$ be the Yang-Baxter element associated to the longest
length permutation in $\mathfrak{S}_n$.

We will show that the element $Y_n$ provides the one-dimensional sign representation
$\rho$ of the BMW algebra $B_n(r,q)$ \cite[Section 3]{HuXiao}, which is defined on
generators by
$$
\rho(T_i)=-q^{-1},\quad \rho(E_i)=0.
$$

\begin{lemma}\label{3.6}
The Yang-Baxter element $Y_n$ satisfies
$$
bY_n=\rho(b)Y_n=Y_nb,
$$
for any $b\in B_n(r,q)$. Furthermore $Y_n$ is a central idempotent.
\end{lemma}

\begin{proof}
For each $1\leq j<n$ there exists a reduced expression for the longest
length permutation such that the reduced expression begins with $s_j$.
Then $Y_n$ has left factor $Y_j(1)$. We have
$$\begin{aligned}
-[2]T_jY_j(1)&=T_j\left(T_j-q+\frac{q-q^{-1}}{1+rq^{-1}}E_j  \right)\\
&=1+(q-q^{-1})(T_j-E_jT_j)-qT_j+\frac{q-q^{-1}}{1+rq^{-1}}T_jE_j\\
&=1-q^{-1}T_j+(q-q^{-1})\frac{-q^{-1}}{1+rq^{-1}}E_j\\
&=(-q)^{-1}(-[2]Y_j(1)).
\end{aligned}$$
A similar calculation shows $E_jY_j(1)=0$, and hence $bY_n=\rho(b)Y_n$.
The equation $\rho(b)Y_n=Y_nb$ for any $b\in B_n(r,q)$ can be proved
similarly. Thus $Y_j(k)Y_n=Y_n$ for any $k\in \mathbb{N}$ and this
yields $Y_n$ is a central idempotent.
\end{proof}

\begin{proposition}\label{3.7}
Let $r=-q^{2m+1}$. We have $F(Y_{m+1})=0$.
\end{proposition}

\begin{proof}
It follows from Theorem \ref{2.11} and Lemma \ref{3.6} that $F(Y_{m+1})$ is an idempotent in
$\End_{U_q(\mathfrak{sp}_{2m})}(V^{\otimes (m+1)})$.
Since the RT-functor $F$ is pivotal, it preserves the traces of morphisms (see \cite[P.160]{FY}).
The rank of an idempotent is equal to its trace and hence it is sufficient to
show that the trace of $Y_{m+1}$ equals zero.

The trace $\tr_n$ of an endomorphism $D\in \mathcal{D}(n,n)$ for any $n\in \mathbb{N}$
in the pivotal category $\mathbb{D}(\mathfrak{sp}_{2m})$ is defined by the evolution
and the coevolution as follows
$$
\tr_n(D):=A_{2n}\circ (D\otimes I_n)\circ U_{2n}.
$$
The following properties can be verified directly
$$\begin{aligned}
\tr_{n+1}(D\otimes I)&=x\tr_n(D),\\
\tr_{n+1}((D\otimes I)\circ T_n\circ (D'\otimes I))&=r\tr_n(D\circ D'),\\
\tr_{n+1}((D\otimes I)\circ E_n\circ (D'\otimes I))&=\tr_n(D\circ D'),
\end{aligned}$$
for any tangles $D,D'\in \mathcal{D}(n,n)$, where $T_n, E_n$ are identified with
tangles in $\mathcal{D}(n+1,n+1)$ by Theorem \ref{2.9}.

A particular choice of the reduced expression for the longest length permutation gives
$$
Y_{m+1}=Y_1(1)Y_2(2)\cdots Y_m(m)Y_m.
$$
Hence, we have
$$\begin{aligned}
&\quad \tr_{m+1}(Y_{m+1})\\
&=\frac{-1}{[m+1]}\tr_{m+1}\left(Y_1(1)Y_2(2)\cdots Y_{m-1}(m-1)
\left([m]T_m-q^m+\frac{q^m-q^{-m}}{1+rq^{-2m+1}}E_m\right)Y_m \right)\\
&=\frac{-1}{[m+1]}\left([m]r-q^mx+\frac{q^m-q^{-m}}{1+rq^{-2m+1}}\right)
\tr_m(Y_1(1)Y_2(2)\cdots Y_{m-1}(m-1)Y_m).
\end{aligned}$$
When $r=-q^{2m+1}$, $x=1-[2m+1]$ and then $\tr_{m+1}(Y_{m+1})=0$.
\end{proof}

\begin{definition}\label{3.8}
Let $\langle Y_{m+1}\rangle$ be the subspace of $\oplus_{s,t}\mathcal{D}(s,t)$
spanned by the morphisms in $\mathbb{D}(\mathfrak{sp}_{2m})$ obtained from
$Y_{m+1}$ by composition and tensor product. Set $\langle Y_{m+1}\rangle^s_t:=
\langle Y_{m+1}\rangle \cap\mathcal{D}(s,t)$.
\end{definition}

The FFT and SFT of invariants for the quantized symplectic group can
be respectively interpreted as Parts (i) and (ii) of the following
theorem, which is the quantum version of \cite[Theorem 4.8]{LZ2} in the symplectic case.

\begin{theorem}\label{3.9}
With notations as above and $r=-q^{2m+1}$, we have

\begin{enumerate}
\item[(i)] the RT-functor $F: \mathbb{D}(\mathfrak{sp}_{2m})\rightarrow
\mathcal{T}(V)$ is full, i.e. $F$ is surjective on $\Hom$ spaces;

\item[(ii)] the map $F^s_t: \mathcal{D}(s,t)\rightarrow H(s,t)$ is injective
if $s+t\leq 2m$, and $\Ker(F^s_t)=\langle Y_{m+1}\rangle^s_t$ if $s+t>2m$.
\end{enumerate}
\end{theorem}

\begin{proof}
It is clear that the theorem is true when $s+t$ is odd. Now assume
that $s+t=2n$ for some $n\in \mathbb{N}$. It follows from Lemma \ref{2.12} that
we have a canonical isomorphism $\mathcal{D}(s,t)\cong \mathcal{D}(n,n)$,
and the study of $F^s_t$ is equivalent to that of $F^n_n$.

(i). The surjectivity of $F^n_n$ follows from the Schur-Weyl duality of type $C$,
see \cite[10.2]{CP}, or \cite{Ha,Hu} for example.

(ii). When $2n\leq 2m$, i.e. $n\leq m$, the injectivity of $F^n_n$ also follows from
the Schur-Weyl duality of type $C$, see \cite[Theorem 5.5]{Hu} for example.
When $m<n$, we set $Y_{m+1}^{(n)}:=Y_{m+1}\otimes I_{n-m-1}\in B_n(-q^{2m+1},q)$.
Hence by Proposition \ref{3.7} it is sufficient to show
$$
\dim B_n(-q^{2m+1},q)/(Y_{m+1}^{(n)})\leq\dim H(n,n)=\dim \End_{U_q(\mathfrak{sp}_{2m})}(V^{\otimes n}),
$$
where $(Y_{m+1}^{(n)})$ means the ideal of $B_n(-q^{2m+1},q)$ generated by $Y_{m+1}^{(n)}$.
In fact
$$\begin{aligned}
\dim (Y_{m+1}^{(n)})&\geq \dim_{\mathbb{Q}}(Y_{m+1}^{(n)}\downarrow_{q=1})\\
&=\dim B_n(-q^{2m+1},q)-\dim
\End_{U_q(\mathfrak{sp}_{2m})}(V^{\otimes n}),
\end{aligned}$$
where the first inequality follows from the specialization $\lim_{q\rightarrow 1}$,
see the Part (i) of \cite[Theorem 8.2]{LZ2}.
This completes the proof of the theorem.
\end{proof}

\begin{remark}
Theorem \ref{3.9} with $s=2n, t=0$ yields the linear formulation of fundamental theorems
for the quantized symplectic group,
while the endomorphism algebra formulation arises from the case $s=t=n$. The equivalence
of these two versions is a consequence of Lemma \ref{2.12}.
\end{remark}

\noindent{\bf Acknowledgements}

We would like to thank the referees for their valuable comments and suggestions.

\end{document}